\documentclass{amsart}

\usepackage{amssymb}
\usepackage{amsmath}

\title{Existence of $q$-Analogs of Steiner Systems}

\author[M.\,Braun]{Michael Braun}
\address{%Faculty of Computer Science\\
Darmstadt University of Applied Sciences\\
Darmstadt, Germany}
\email{michael.braun@h-da.de}

\author[T.\,Etzion]{Tuvi Etzion}
\address{%Department of Computer Science\\
Technion, Haifa, %32000,
Israel}
\thanks{The research of T.\,E. was supported in part by the Israeli
Science Foundation (ISF),~Jerusa\-lem, Israel, under Grant 10/12.}
\email{etzion@cs.technion.ac.il}

\author[P.R.J.\,\"Osterg{\aa}rd]{Patric R.\,J.\ \"Osterg{\aa}rd}
\address{%Department of Communications and Networking\\
Aalto University\\ % School of Electrical Engineering\\
Aalto, Finland}
\thanks{The research of P.\,R.\,J.\,\"O. was
supported in part by the Academy of Finland
under Grant No.\ 132122.}
\email{patric.ostergard@aalto.fi}

\author[A.Vardy]{Alexander Vardy}
\address{%Department of Electrical and Computer Engineering\\
University of California San Diego\\
La Jolla, CA} % 92093, USA}
\email{avardy@ucsd.edu}

\author[A.\,Wassermann]{Alfred Wassermann\vspace*{7ex}}
\address{%Department of Mathematics\\
University of Bayreuth\\
Bayreuth, Germany}
\email{Alfred.Wassermann@uni-bayreuth.de}

\newtheorem{theorem}{Theorem}
\newtheorem{lemma}[theorem]{Lemma}
\newtheorem{conjecture}[theorem]{Conjecture}

\DeclareMathOperator{\GL}{GL}
\DeclareMathOperator{\LL}{L}
\DeclareMathOperator{\Gal}{Gal}

\DeclareMathOperator{\GF}{GF}
\DeclareMathOperator{\Aut}{Aut}
\DeclareMathOperator{\inv}{inv}

\renewcommand{\leq}{\leqslant}
\renewcommand{\geq}{\geqslant}

\newcommand{\F}{\mathbb{F}}
\newcommand{\Z}{\mathbb{Z}}
\newcommand{\GammaL}{\Gamma\!\LL}

\makeatletter
\DeclareRobustCommand{\sbinom}{\genfrac[]\z@{}}
\makeatother
\newcommand{\G}[2]{\sbinom{{#1}}{{#2}}}

\newcommand{\deff}{\mbox{$\stackrel{\rm def}{=}$}}

\newcommand{\dS}{\mathcal S} % Steiner system
\newcommand{\zero}{{\mathbf 0}}
\newcommand{\al}{\alpha}
\newcommand{\bM}{{\mathbf M}}

\begin{document}

\begin{abstract}
%\noindent\looseness=-1
Let $\F_q^n$ be a vector space of dimension $n$ over the finite field $\F_q$.
A~\mbox{$q$-analog} of a Steiner system (briefly, a $q$-Steiner system), denoted
$\dS_q[t,k,n]$, is a~set $\dS$ of $k$-dimensional subspaces of $\F_q^n$ such
that each $t$-dimensional subspace of $\F_q^n$ is contained in exactly one
element of $\dS$. Presently, $q$-Steiner systems are known only for $t=1$, and
in the trivial cases $t = k$ and $k = n$. In %\linebreak
this paper, the first nontrivial
$q$-Steiner systems with $t\;{\geq}\;2$ are constructed.
Specifically, several nonisomorphic
$q$-Steiner systems $\dS_2[2,3,13]$ are found by requiring that
their automorphism groups contain the normalizer of a~Singer subgroup of
%R the general linear group
$\GL(13,2)$. This approach leads to an instance of the exact cover
problem, which turns out to have many solutions.%\vspace{-2.50ex}
\end{abstract}

\maketitle
%\vspace{10ex}

\thispagestyle{empty}

%==============================================================================%
%                                                                              %
%   1. INTRODUCTION                                                            %
%                                                                              %
%==============================================================================%

\section{Introduction} \label{sec:introduction}

Let $V$ be a set with $n$ elements. A \emph{$t$-$(n,k,\lambda)$ combinatorial
design} (or {\emph{$t$-design}}, in brief)
is a collection of $k$-subsets of $V$,
called blocks, such that each $t$-subset of $V$ is contained in exactly
$\lambda$ blocks. A $t$-$(n,k,1)$ design with $t \geq 2$ is called a
\emph{Steiner system}, and usually denoted $S(t,k,n)$.
Nontrivial $t$-designs exist for
all $t$, while nontrivial (meaning $t < k < n$) Steiner systems are known to
exist for $2 \leq t \leq 5$; see~\cite[Part II]{CD07}, for example.

\looseness=-1
It was suggested by Tits~\cite{Tits57} in 1957 that combinatorics of sets could
be regarded as the limiting case $q \to 1$ of combinatorics of vector spaces
over the finite field $\F_q$. Indeed, there is a strong analogy between subsets
of a set and subspaces of a vector space, expounded by numerous
authors---see~\cite{Cohn,GR,Wang} and references therein. In particular, the
notions of $t$-designs and Steiner systems have been extended to vector spaces
by Cameron~\cite{Cameron1,Cameron2} and Delsarte~\cite{Delsarte} in the early
1970s. Specifically, let $\F_q^n$ be a vector space of dimension $n$ over the
finite field $\F_q$.
Then a \emph{$t$-$(n,k,\lambda)$ design over $\F_q$}
is a collection of $k$-dimensional~subspaces of $\F_q^n$
($k$-subspaces, for short), called blocks, such that each
$t$-subspace of $\F_q^n$ is contained in exactly $\lambda$ blocks.
Such $t$-designs over $\F_q$ are the $q$-analogs of conventional
combinatorial
designs.
A~$t$-$(n,k,1)$ design over $\F_q$ is said to
be a \emph{$q$-Steiner system}, and denoted $\dS_q[t,k,n]$.

The first examples of nontrivial $t$-designs over $\F_q$ with $t \geq 2$ were
found~by~Tho\-mas~\cite{T87} in 1987. Today, following the work of many
authors~\cite{BKL,MMY,Chaudhuri,S90,S92,T96}, numerous such examples are known.

The situation is very different for $q$-Steiner systems. They are known to exist
in the trivial cases $t = k$ or $k=n$, and in the case where $t = 1$ and $k$
divides $n$. In the latter case, $q$-Steiner systems coincide with the classical
notion of \emph{spreads} in projective geometry~\cite[Chap.\ 24]{LW}.
Beutelspacher~\cite{Beutelspacher} asked in~1978 whether nontrivial $q$-Steiner
systems with $t \geq 2$ exist, and this question has tantalized mathematicians
ever since. The problem has been studied by numerous
authors~\cite{AAK,EV11b,M,SE,T87,T96}, without much
progress toward constructing
such $q$-Steiner systems. In particular, Thomas~\cite{T96} showed in 1996 that
certain kinds of $\dS_2[2,3,7]$ $q$-Steiner systems (the smallest possible
example) cannot exist. In 1999, Metsch~\cite{M} conjectured that~nontrivial
$q$-Steiner systems with $t \geq 2$ do not exist in general.

Our main result is the following theorem.
\begin{theorem}
\label{main}
There exist nontrivial $q$-Steiner systems with $t \geq 2$.
\end{theorem}

\looseness=-1
In fact, we have discovered over 400 nonisomorphic $\dS_2[2,3,13]$ $q$-Steiner
systems. For more on this, see Section~\ref{sec:sol}; however, let us briefly
outline~our~general approach here. We begin by imposing a carefully chosen
\emph{additional structure} on a putative $\dS_2[2,3,13]$ $q$-Steiner system
$\dS$. Specifically, we construct~a group $A \le \GL(13,2)$ as the semidirect
product of the Galois group $\Gal(\F_{2^{13}}/\F_2)$ and a Singer subgroup
$C_{\alpha}$ of $\GL(13,2)$, and then insist that the automorphism group
$\Aut(\dS)$ contain $A$. Next, we make use of the well-known Kramer--Mes\-ner
method, and consider the Kramer--Mesner incidence structure between
the orbits of $2$-subspaces of $\F_2^{13}$ and the orbits of $3$-subspaces of
$\F_2^{13}$ under the action of $A$. Given the corresponding Kramer--Mesner
matrix ${\mathbf M}^A$ with 105 rows and 30\,705 columns, we reformulate the
search for $\dS$ as an instance of the exact cover problem, which we solve using
the dancing links algorithm of Knuth. %~~\cite{KP,K00}.

\looseness=-1
As corollaries to the existence of $\dS_2[2,3,13]$,
% Theorem\,\ref{main},
we obtain a number of related results.
Starting with $\dS_2[2,3,13]$, we use
\cite[Theorem\,3.2]{EV11b} to construct
a~Steiner system $S(3,8,8192)$.
Steiner systems with these parameters were not known previously.
%R~\cite{CD07}.
%Theorem\,\ref{main} also produces
%We also find
$\dS_2[2,3,13]$ also leads to
new diameter-perfect codes in the Grassmann graph~\cite{AAK,SE}.
Finally, we note that $q$-Steiner systems have applications for
error-correction in networks, under randomized network coding,
as shown in~\cite{EV11a,KK08a}. Thus we find that the maximum
number of codewords in a subspace code over $\F_2^{13}$
of constant dimension $k = 3$ and minimum subspace distance $d=4$
is $1\,597\,245$.

\looseness=-1
The rest of this paper is organized as follows. In Section~\ref{sec:map}, we
consider automorphisms of $q$-Steiner systems,\kern-0.5pt\ and introduce the
normalizer of a Singer~sub\-group, which is the group of automorphisms we
choose to impose on $\dS_2[2,3,13]$. In Section~\ref{sec:sol},
we briefly outline
the Kramer--Mesner method, and describe the computer search we have carried out
based upon the results of Section~\ref{sec:map}. We give an explicit set
of 15 orbit
representatives for the 1\,597\,245 subspaces of one $\dS_2[2,3,13]$
$q$-Steiner system, thereby proving Theorem~\ref{main}. We also present
several negative results that establish nonexistence of $q$-Steiner systems of
certain kinds, extending the work of \cite{EV11b,KK08b,T96}.
We discuss the connection
to difference sets in Section~\ref{sec:related},
and compile a number of related results.
In Section~\ref{sec:conc}, we conclude with a brief discussion of open
problems, and formulate a specific conjecture regarding the existence
of an infinite family of $q$-Steiner systems.

\newpage
\vspace{1.00ex}
%==============================================================================%
%                                                                              %
%   2. Automorphisms of $q$-Steiner systems                                    %
%                                                                              %
%==============================================================================%
\section{Automorphisms of $q$-Steiner systems}
\label{sec:map}

\looseness=-1
Let $G$ be the group of bijective incidence-preserving
mappings from~the~set of subspaces of $\F_q^n$ onto itself.
We know from the
fundamental theorem of projective geometry~\cite[Chap.\,3]{Bae52}
that $G$ is the general semilinear group $\GammaL(n,q)$. This
group is isomorphic to the semidirect product of the general
linear group $\GL(n,q)$ and the Galois group $\Gal(\F_q/\F_p)$,
where $p$ is the characteristic of $\F_q$. Unless stated
otherwise, we will henceforth assume that $q$ is prime, in
which case $G$ reduces %R restricts
to the general linear group $\GL(n,q)$.
Basic familiarity with the main properties of
$\GL(n,q)$ is assumed; an in-depth treatment of this group can
be found, for example, in \cite{H67}.

\looseness=-1
The action of $\GL(n,q)$ on subspaces of $\F_q^n$ extends in
the obvious way to sets of subspaces, and thereby to $q$-Steiner systems.
Given a set $\dS$ of subspaces of $\F_q^n$ and a group element
$g\in \GL(n,q)$, we denote the image of $\dS$~under the
action of $g$ by $\dS^g$. We say that two
sets of subspaces $\dS_1$ and $\dS_2$ are \emph{isomorphic} if there exists an
element $g \in \GL(n,q)$ such that $\dS_2 = \dS_1^g$.
An element $g \in \GL(n,q)$ for~which
$\dS^g = \dS$ is called an \emph{automorphism} of $\dS$. The automorphisms
of a set $\dS$ of subspaces form a group under composition, called the
\emph{automorphism group} and denoted $\Aut(\dS)$. A subgroup of $\Aut(\dS)$
is called a~\emph{group of automorphisms}. We note that,
since $\GL(n,q)$ acts transitively
on~the
set of $k$-subspaces of $\F_q^n$ for any fixed $k$, the automorphism group
of a nontrivial $q$-Steiner system is necessarily a~proper subgroup
of $\GL(n,q)$.

A well-known approach to constructing combinatorial objects is to prescribe
a~certain group
of automorphisms $A$ and then search only for those objects whose
automorphism group contains $A$. For an overview of the theory and applications
of this method to combinatorial designs, the reader is referred to
%\cite[Sect.~6.3]{GO} and
\cite[Sect.~9.2]{KO}. Prescribing a group of
automorphisms simplifies the construction problem, sometimes rendering
intractable problems tractable, but choosing the right groups can be
a~challenge. We shall now discuss certain apposite subgroups of $\GL(n,q)$.

A \emph{Singer cycle} of $\GL(n,q)$ is an element of order $q^n-1$.
Singer cycles~can be constructed, for example, by identifying
vectors in $\F_q^n$ with elements of the finite field $\F_{q^n}$.
Since multiplication by a primitive element $\alpha \in \F_{q^n}$ is a~linear
operation, it corresponds to a Singer cycle in $\GL(n,q)$.
In fact, there is a~one-to-one correspondence between Singer cycles in
$\GL(n,q)$ and primitive elements in $\F_{q^n}$.
Given a primitive element $\al \in \F_{q^n}$, the
subgroup of $\GL(n,q)$ generated by the corresponding
Singer cycle is cyclic of order $q^n-1$, and its elements
correspond to multiplication by $\al^i$ for $i = 0,1,\ldots,q^n\!-2$.
We denote such groups by $C_\al$ and call them the
\emph{Singer subgroups} of $\GL(n,q)$.

\looseness=-1
Another group of importance to us is generated by the \emph{Frobenius
automorphism} $\phi: \F_{q^n} \to \F_{q^n}$, defined by $\phi(\beta) = \beta^q$
for all $\beta \in \F_{q^n}$. The Frobenius automorphism $\phi$
is the canonical
generator of the \emph{Galois group} $\Gal(\F_{q^n}/\F_q)$,
which is cyclic of order $n$.

The \emph{normalizer of a subgroup $H \le G$} is the set of elements of $G$
that commute with $H$ as a whole. That is, $N_G(H)=\{ g \in G : gH = Hg\}$.
The~following well-known result can be found, for example, in
\cite[pp.~187--188]{H67}. %\vspace{-0.50ex}

\begin{theorem}
\label{order}
\hspace*{-3pt}Let
$A_\alpha$ be the normalizer of a Singer subgroup $C_\alpha$ in\/ $\GL(n,q)$.
Then $A_\alpha$ has order $n(q^n-1)$ and is isomorphic to the semidirect product
of the Galois group\/ $\Gal(\F_{q^n}/\F_q)$ and\/ $C_\alpha$.
\end{theorem}

The following theorem follows from a more general result by Kantor~\cite{K80}.
It is stated explicitly in~\cite{D}. %\vspace{-0.50ex}

\begin{theorem} \label{thm:max}
Let $n$ be an odd prime. Then the normalizer of a Singer~subgroup is a maximal
subgroup of\/ $\GL(n,q)$.
\end{theorem}
%
% REMARK:
% except when $n = q = 2$
%

\looseness=-1
In Section \ref{sec:sol}, we will search for $q$-Steiner
systems $\dS$ whose automorphism group $\Aut(\dS)$ % \le \GL(n,q)$
contains the normalizer of a Singer subgroup.
We observe that Singer~subgroups and normalizers of Singer
subgroups have been used when prescribing automorphisms for
various types of $q$-analog structures~in~\cite{BKL,EV11a}.

We already noted that $\Aut(\dS) < \GL(n,q)$ for
nontrivial designs over $\F_q$.
Thus for odd primes $n$, it follows from Theorem~\ref{thm:max}
that if $\Aut(\dS)$ contains~$A_\alpha$, then $A_\al$ is the
full automorphism group of $\dS$, namely $\Aut(\dS) = A_\alpha$.
In turn, the fact that
$\Aut(\dS) = A_\alpha$ makes it possible to say
much more. In particular, we will show that distinct nontrivial
designs whose automorphism group contains $A_\alpha$ are
necessarily nonisomorphic. First, we need the following
lemma.\vspace{-0.50ex}

\begin{lemma}
\label{lem:self}
%Let $A_\alpha$ be the normalizer of a Singer subgroup $C_\alpha$
The normalizer $A_\alpha$
of a Singer subgroup %$C_\alpha$
is self-normalizing in\/ $\GL(n,q)$. That
is, $A_\alpha = N_{\GL(n,q)}(A_\alpha)$.\vspace{-0.50ex}
\end{lemma}

\begin{proof}
Let $g \in N_{\GL(n,q)}(A_\alpha)$. Then
$g^{-1}C_\alpha g \le g^{-1}A_\alpha g = A_\alpha$.
The conjugate of a Singer subgroup is also a Singer subgroup.
On the other hand,~it~is known \cite[Proposition\,2.5]{CR04}
that $A_\alpha$ contains a \emph{unique} Singer subgroup.
In conjunction with $g^{-1}C_\alpha g \le A_\alpha$, this implies
that $g^{-1}C_\alpha g = C_\alpha$.
This, in turn, implies that $g \in A_\alpha$,
and therefore $N_{\GL(n,q)}(A_\alpha) = A_\alpha$.
\end{proof}

\begin{theorem}
\label{thm:iso2}
Suppose that $n \geq 3$ and $q$ are primes,
and let $A_\alpha$ be the~norm\-alizer
of a Singer subgroup in $\GL(n,q)$.
Then distinct nontrivial $q$-Steiner~systems
$\dS_q[t,k,n]$ admitting $A_\alpha$ as a group of
automorphisms are nonisomorphic.
\end{theorem}

\begin{proof}
Let $\dS_1$ and $\dS_2$ be two distinct $\dS_q[t,k,n]$
$q$-Steiner systems, both~admitting $A_\alpha$ as a group of
automorphisms. Then
\begin{equation}
\label{aux1}
\Aut(\dS_1) \,=\, \Aut(\dS_2) \,=\, A_\alpha
\end{equation}
by Theorem~\ref{thm:max}. % Now %suppose
Assume to the contrary that $\dS_1$ and $\dS_2$
are isomorphic, that is $\dS_1^g = \dS_2$ for some $g \in \GL(n,q)$.
Let $a \in A_{\alpha}$. Then it follows from (\ref{aux1}) that
\begin{equation}
\label{aux2}
\dS_2^{g^{-1}ag} =\, \dS_1^{ag} =\, \dS_1^g \,=\, \dS_2
\end{equation}
and therefore $g^{-1}ag \in \Aut(\dS_2) = A_{\alpha}$.
Since (\ref{aux2}) holds for all $a \in  A_{\alpha}$,
we~conclude that $g$ must belong to the normalizer of $A_{\alpha}$,
which is $A_\alpha$ itself by Lemma~\ref{lem:self}.
But for $g \in A_\alpha$, we have $\dS_1^g = \dS_1$ by (\ref{aux1}).
Hence $\dS_2 = \dS_1$, a contradiction.
\end{proof}

We next show how to classify the subspaces of $\F_q^n$ into orbits under
the~action of various groups of interest. Fix a primitive element $\alpha$ of
$\F_{q^n}$, and write a $k$-subspace $X$ of $\F_q^n$ as
$X = \{\zero,\alpha^{x_1},\alpha^{x_2},\ldots,\alpha^{x_m}\}$, where $m = q^k-1$
and $x_1,x_2,\ldots,x_m \in \Z_{q^n-1}$. For $x \in \Z_{q^n-1}$, let $\rho(x)$
be the minimal cyclotomic~representative for $x$,~that is
$\rho(x) = \min\{ xq^i \mod (q^n-1) ~:~ 0\leq i \leq n-1\}$.
We define
\begin{align*}
\inv_F(X)~\,\deff~&
~\{ \rho(x_i) ~:~ 1\leq i \leq m\},
\\[-.30ex]
\inv_S(X)~\,\deff\,~&
~\{ x_i - x_j ~:~ 1\leq i,j \leq m \text{ with $i \ne j$}\},
\\[-0.30ex]
\inv_N(X)~\,\deff\,~&
~\{ \rho(x_i - x_j) ~:~ 1\leq i,j\leq m \text{ with $i \ne j$}\}.
\end{align*}

%\vspace{0.50ex}
\begin{lemma}
\label{lem:orbit} $\,$\vspace{.50ex}
\begin{enumerate}
\item
If two $k$-subspaces $X$, $Y$ of\/ $\F_q^n$ are in the same orbit
under the action of the Galois group\/ $\Gal(\F_{q^n}/\F_q)$ then\/
$\inv_F(X) = \inv_F(Y)$.

\item
If two $k$-subspaces $X$, $Y$ of\/ $\F_q^n$ are in the same orbit
under the action of the Singer subgroup $C_\alpha$ then\/
$\inv_S(X) = \inv_S(Y)$.

\item
\label{lem:orbit3}
If two $k$-subspaces $X$, $Y$ of\/ $\F_q^n$ are in the same orbit
under the action of the normalizer $A_\alpha$ of the Singer subgroup
$C_\alpha$ then\/ $\inv_N(X) = \inv_N(Y)$.
\end{enumerate}
\end{lemma}

\begin{proof}
Let $X = \{\zero,\alpha^{x_1},\alpha^{x_2},\ldots,\alpha^{x_m}\}$
be a $k$-subspace of $\F_q^n$, %as before,
with
$x_1,x_2,\ldots,x_m$ in $\Z_{q^n-1}$.
The action of the generator of $C_\alpha$ on $X$
increases $x_1,x_2,\ldots,x_m$ by one
modulo $q^n-1$, thereby preserving the differences
between them.
The action of the Frobenius automorphism $\phi$ on $X$ multiplies
each $x_i$ by $q$ modulo $q^n-1$, thereby leaving it in the
same cyclotomic coset.
\end{proof}

\vspace{1ex}
%==============================================================================%
%                                                                              %
%   3. Main result                                                             %
%                                                                              %
%==============================================================================%
\section{Kramer--Mesner computer search}
\label{sec:sol}

Constructing  $t$-designs over $\F_q$ is equivalent to
%R finding solutions to
solving
certain systems of linear Diophantine equations.
Let $\bM$ be a $\{0,1\}$ matrix with rows and columns
indexed by the $t$-subspaces and the $k$-subspaces of $\F_q^n$,
respectively; there is a~$1$ in row $X$ and column $Y$ of $\bM$
iff $t$-subspace $X$ is contained in $k$-subspace~$Y$.
With this definition, % of $\bM$,
a $t$-$(n,k,\lambda)$ design
over $\F_q$ is precisely a $\{0,1\}$ solution to
\begin{equation}
\label{Mx}
\bM \,x \ = \ (\lambda,\lambda,\ldots,\lambda)^T .
\end{equation}
Unfortunately, for most parameters of interest,
finding solutions to the resulting large systems
of equations is outside the realm of computational
feasibility.

\looseness=-1
However, if we impose a prescribed group of automorphisms $A$
on a putative solution, thereby reducing the size of
the problem, the situation can still be described
in terms of a system of linear equations. In this case,
the rows and columns of the matrix $\bM^A$ correspond
to $A$-orbits of $t$-subspaces and $k$-subspaces; the
entries of $\bM^A$ are nonnegative integers, possibly
greater than~$1$. This %technique
is analogous to a well-known technique in design theory that
is called the
{Kramer--Mesner method} after its developers~\cite{KM}.
For more details on applications of the Kramer--Mesner method
in the context of designs over $\F_q$, see~\cite{BKL}.

There are several %generic
group-theoretic algorithms that,
given a prescribed group $A$ acting on a set of finite structures,
% of finite structures,
compute the orbits under~$A$ and produce the corresponding
Kramer--Mesner matrix. For more details, see \cite{BKL,Schmalz}.
In our case,~the Kramer--Mesner matrix $\bM^A$, where $A$ is
the normalizer of a Singer subgroup, can be also computed directly
using the invariants in Lemma~\ref{lem:orbit}.

In order to find a solution to (\ref{Mx}) for a given
Kramer--Mesner matrix $\bM^A$, we observe that
when $\lambda = 1$, the system of equations in (\ref{Mx})
reduces to an instance of the \emph{exact cover} problem~\cite{K00}.
That is, we wish to find a set $\mathcal{S}$ of columns of~$\bM^A$
such that for each row of $\bM^A$, there is exactly one column
of~$\mathcal{S}$~containing $1$ in this row. The exact cover
problem can be solved efficiently using the dancing links algorithm
of Knuth. For more on this, see~\cite{KP,K00}.

\looseness=-1
We now specialize the above to the case of the $q$-Steiner system
$\dS_2[2,3,13]$.
%To begin with,
%To start,
At first,
the matrix $\bM$ in (\ref{Mx}) has
$\smash{\G{13}{2}} = 11\,180\,715$ rows and
$\smash{\G{13}{3}} = 3\,269\,560\,515$ columns,
where $\G{n}{k}$ is the $q$-binomial coefficient with $q = 2$.
We need to find an exact cover of the rows of $\bM$ consisting
of some
$|\dS_2[2,3,13]| = \smash{\G{13}{2}/\G{3}{2}} = 1\,597\,245$ columns.
However,
the resulting instance of the exact cover problem
is well beyond the domain of feasibility of existing algorithms.
Instead, we prescribe the normalizer $A_\al$ of a Singer subgroup
of $\GL(13,2)$ as a group of automorphisms.
Specifically, we have used~the Singer subgroup generated by
the primitive element $\alpha \in \F_{2^{13}}$ which
is a root of the polynomial $x^{13}+x^{12}+x^{10}+x^9+1$.
%However,
We note that the specific choice of the primitive element
is unimportant in the sense that the set of Singer subgroups
(and, thereby, also the set of their normalizers) forms a conjugacy
class of subgroups of $\GL(n,q)$. %\cite[Chap.\ II]{H67}.
By Theorem~\ref{order}, we have
$|A_\al| = 13\bigl(2^{13}-1\bigr) = 106\,483$.
The \mbox{orbits} of $2$-subspaces and $3$-subspaces
under the action of $A_\al$ are all
full-length, resulting in a Kramer--Mesner matrix $\bM^{A_\al}$
with %only
$\smash{\G{13}{2}}/106\,483 = 105$ rows~and
$\smash{\G{13}{3}}/106\,483 = 30\,705$ columns.
As all the orbits have full length $|A_\al|$,~we~need
to find an exact cover consisting of %only
$|\dS_2[2,3,13]|/|A_\al| = 15$
columns~of~$\bM^{A_\al}$.
One such sets of columns~corresponds to the $15$ subspaces of $\F_2^{13}$ listed
below:\vspace{0.30ex}
%
% REMARK:
% actually, it would suffice to give three generators for each or two
%
\begin{equation}
\label{eq:sol}
\renewcommand{\arraystretch}{1.10}
\begin{tabular}{@{}l@{~}l@{}}
$\{0,1,1249,5040,7258,7978,8105\}$,&
$\{0,7,1857,6681,7259,7381,7908\}$,\\
$\{0,9,1144,1945,6771,7714,8102\}$,&
$\{0,11,209,1941,2926,3565,6579\}$,\\
$\{0,12,2181,2519,3696,6673,6965\}$,&
$\{0,13,4821,5178,7823,8052,8110\}$,\\
$\{0,17,291,1199,5132,6266,8057\}$,&
$\{0,20,1075,3939,3996,4776,7313\}$,\\
$\{0,21,2900,4226,4915,6087,8008\}$,&
$\{0,27,1190,3572,4989,5199,6710\}$,\\
$\{0,30,141,682,2024,6256,6406\}$,&
$\{0,31,814,1161,1243,4434,6254\}$,\\
$\{0,37,258,2093,4703,5396,6469\}$,&
$\{0,115,949,1272,1580,4539,4873\}$,\\
\multicolumn{2}{c}{$\{0,119,490,5941,6670,6812,7312\}$.}\\[0.30ex]
\end{tabular}
\end{equation}
Each $3$-subspace
$\{\zero,\alpha^{x_1},\alpha^{x_2},\ldots,\alpha^{x_7}\}$
is specified in (\ref{eq:sol}) in terms of the exponents
$\{x_1,x_2,\ldots,x_7\}$
of its nonzero elements. The
$A_\al$-orbits of the $15$ subspaces in (\ref{eq:sol})
%R under the action of $A_\al$
form a $q$-Steiner system $\dS_2[2,3,13]$,
thereby proving Theorem~\ref{main}.

The first solution to the exact cover problem instantiated by
the Kramer--Mesner matrix $\bM^{A_\al}$ was found in about two
hours on a personal computer. After about a month, we have
found $401$ distinct solutions. By Theorem~\ref{thm:iso2}, these
solutions give rise to $401$ nonisomorphic
$\dS_2[2,3,13]$ $q$-Steiner systems.
%Unfortunately,
We note, however, that
classifying \emph{all} the solutions to the exact
cover instance specified by $\bM^{A_\al}$ does not
appear to be computationally feasible.
%
% REMARK
% computation beyond one million years
%

\looseness=-1
Inspired by the positive results for $\dS_2[2,3,13]$,
we have searched extensively for other $q$-Steiner systems,
with various %different
parameters, while imposing
certain groups of automorphisms. We were able to
resolve definitively~the~seven %different
cases
listed below. Unfortunately, in all these cases the
outcome was negative. Our results show that
$q$-Steiner systems with the following parameters
and automorphisms do not exist:\vspace{0.15ex}
\begin{equation}
\label{nonexistence}
\begin{tabular}{r@{\,~}l}
%$\dS_2[2,3,7]$, & Singer subgroup (order 127)\\[-0.05ex]
$\dS_2[2,3,7]$, & Galois group $\Gal(\F_{2^7}/\F_2)$ (order 7)\\[0.05ex]
$\dS_2[3,4,8]$, & Singer subgroup (order 255)\\[0.05ex]
$\dS_2[2,4,10]$, & normalizer of Singer subgroup (order 10\,230)\\[0.05ex]
$\dS_2[2,4,13]$, & normalizer of Singer subgroup (order 106\,483)\\[0.05ex]
$\dS_2[3,4,10]$, & normalizer of Singer subgroup (order 10\,230)\\[0.05ex]
$\dS_3[2,3,7]$, & Singer subgroup (order 2\,186)\\[0.05ex]
$\dS_5[2,3,7]$, & normalizer of Singer subgroup (order 546\,868)\\[0.05ex]
\end{tabular}
\end{equation}
%do not exist.
This extends upon the previous work on
nonexistence of $q$-Steiner systems~\cite{EV11b,KK08b,T96}.
%In particular,
For example, %In fact,
it was shown in~\cite{KK08b} that
a $q$-Steiner~system $\dS_2[2,3,7]$ admitting
a Singer subgroup
as a group of automorphisms does not exist.
%R The %other
%R nonexistence results in (\ref{nonexistence}) are all
%R new, to the best of our knowledge.

\vspace{2.50ex}
%==============================================================================%
%                                                                              %
%   4. Further Results                                                         %
%                                                                              %
%==============================================================================%
\section{Related results}
\label{sec:related}

\looseness=-1
Obviously, an $\dS_2[2,3,13]$ $q$-Steiner system gives
rise to an~$S(2,7,8191)$~Stei\-ner system: simply represent each
subspace of $\F_2^{13}$ by the characteristic vector of the set
of its nonzero elements.
%R Using the results of~\cite{SE}, we
%R further obtain a Steiner system $S(2,4,4096)$.
We observe
that Steiner systems with these~parameters\linebreak were already
known~\cite[Table 3.3]{AG}.
Notably, however, it follows
from~\cite[The\-o\-rem 3.2]{EV11b}
that $\dS_2[2,3,13]$
also gives rise to an $S(3,8,8192)$ Steiner system.
%This Steiner system is new; in fact,
No $S(3,2^k,2^n)$ Steiner systems with $2^k \geq 8$ were
previously known~\cite{CD07,EV11b}.
The new $S(3,8,8192)$ Steiner~system can be used in various
constructions (e.g., those of \cite{BJL,Blanchard,CD07,H79}) to produce
new $S(3,8,v)$ Steiner systems for many other values of $v$.

Following~\cite{EV11a,KK08a}, we let $\mathcal{A}_q(n,d,k)$ denote
the size of the largest subspace code in $\F_q^n$ of constant
dimension $k$ and minimum subspace distance~$d$. Then
the existence of $\dS_2[2,3,13]$ implies that
$\mathcal{A}_2(13,4,3) = 1\,597\,245$ (the upper~bound
$\mathcal{A}_2(13,4,3) \leq 1\,597\,245$ follows
from~\cite[Theorem\,1]{EV11a}). %, for instance).

The new $\dS_2[2,3,13]$ $q$-Steiner systems found in
Section~\ref{sec:sol} also produce new diameter-perfect
codes in the corresponding Grassmann graph.
Precious~few~examples of such codes are known. For more on
the connection between $q$-Steiner systems and
diameter-perfect codes in a Grassmann graph, see
\cite{AAK,SE}.

In the remainder of this section, we expound upon the
connection between % $\dS_2[2,k,n]$
$q$-Steiner systems and difference families.
Recall from~\cite{AB} that a \emph{$(v,k,\lambda)$ difference family}
over an additive group $G$ of order $v$ % $|G| = v$
is a collection $B_1,B_2,\ldots,B_s$ of $k$-subsets
of $G$ such that every nonidentity element of $G$ occurs
exactly $\lambda$ times in the
multiset $\bigl\{a-b ~:~ a,b \in B_i,\ a \neq b,\ 1 \leq i \leq s\bigr\}$.
\begin{theorem}
\label{thm:connect}\looseness=-1
\hspace*{-2pt}Let $k$ and $n$ be %relatively
coprime, and
suppose there exists an $\dS_2[2,k,n]$ \mbox{$q$-Steiner} system
%with
admitting
a Singer subgroup $C_\al$ as a group of automorphisms.
Then there exists a $(2^n-1,2^k-1,1)$ difference
family over the group\/ $\Z_{2^n-1}$.
\end{theorem}

\begin{proof}
\looseness=-1
Fix a primitive element $\al$ of $\F_{2^n}$, and let $\varphi$ be
%a map from the nonzero vectors of $\F_2^n$ onto $\Z_{2^n-1}$
the bijective homomorphism from the multiplicative group
of $\F_{2^n}$ to $\Z_{2^n-1}$
defined by $\varphi(\al^i) = i$. We extend $\varphi$ to sub\-spaces
$X = \{\zero,\alpha^{x_1},\alpha^{x_2},\ldots,\alpha^{x_m}\}$
of $\F_2^n$ in the obvious way, by defining
$\varphi(X) = \{x_1,x_2,\ldots,x_m\} \subseteq \Z_{2^n-1}$. Now let $\dS$
be an $\dS_2 [2,k,n]$ $q$-Steiner system admitting $C_\al$ as a group
of automorphisms. Partition~the~subspaces of $\dS$ into orbits
under the action of $C_\al$. Since $k$ and $n$ are coprime,
all such orbits have full length $|C_\al|$, % = 2^n-1$,
so that the number of orbits is given by
$$
s
\ = \
\frac{|\dS_2[2,k,n]|}{|C_\al|}
\ = \
\frac{2^{n-1}-1}{(2^k-1)(2^{k-1}-1)}
$$
We choose (arbitrarily) one %representative
subspace from each orbit. Let $X_1,X_2,\ldots,X_s$~be~the
resulting set of orbit representatives. We claim
that $\varphi(X_1),\varphi(X_2),\ldots,\varphi(X_s)$ is
a~$(2^n-1,2^k-1,1)$ difference family over $\Z_{2^n-1}$.

Indeed, consider an arbitrary
nonzero element $a \in \Z_{2^n-1}$.
%This element
Observe that $a$ %It
can~be obtained as a difference of two group elements
in exactly $2^n-1$ ways: $(a+i) - i$ for $i = 0,1,\ldots,2^n-2$.
To each pair $\{a+i,i\}$, there corresponds a $2$-subspace
$\{\zero,\al^i,\al^{a+i},\al^i + \al^{a+i}\}$, %R of $\F_2^n$,
and to each such $2$-subspace, there corresponds a~unique
$k$-subspace of $\dS$.
%It is easy to see that
All such $k$-subspaces of $\dS$
are in the same orbit under the action of $C_\al$, and every
$k$-subspace in this orbit contains
$\{\zero,\al^j,\al^{a+j},\al^j + \al^{a+j}\}$
for some~$j$. It follows that $a$ occurs at least once
as a difference of two elements of~$\varphi(X)$, where
$X$ is the representative of the corresponding orbit.
But the total number of differences in the set
$\bigl\{a-b ~:~ a,b \in \varphi(X_i),\ a \neq b,\ 1 \leq i \leq s\bigr\}$
is given by $s(2^k-1)(2^k-2) = 2^n -2$, which completes the proof.
\end{proof}

In fact, the following more general result is true:
if $k$ and $n$ are coprime, %relatively prime,
then
an $\dS_q[2,k,n]$ $q$-Steiner system that admits \pagebreak
a Singer subgroup as a~group of automorphisms
gives rise to a
$\bigl((q^n-1)/(q-1),\,(q^k-1)/(q-1),\,1\bigr)$
difference family over %the group
$\Z_{(q^n-1)/(q-1)}$.
We omit the proof, which~is~similar to
the proof of Theorem~\ref{thm:connect}.

% REMARK:
% A generator of the Frobenius automorphism group corresponds to a
% \emph{multiplier} in the framework of difference families. We cannot
% get arbitrary multipliers as only elements of $\GL(n,2)$ are
% considered. For elements of the Frobenius automorphism group of
% $\F_{q^n}$ we do know that $a + b = c$ implies that
% $a^{q^i} + b^{q^i} = (a+b)^{q^i} = c^{q^i}$.

In order to obtain an $(8191,7,1)$ difference family
from the $15$ sets in (\ref{eq:sol}), first adjoin to
each such set $\{x_1,x_2,\ldots,x_7\}$ the $12$ sets
$\{2^ix_1,2^ix_2,\ldots,2^ix_7\}$ modulo $8191$, for
$i = 1,2,\ldots,12$ (thereby accounting for the
action of the~Galois group). The resulting $15\,{\cdot}\,13 = 195$
sets indeed form an $(8191,7,1)$ difference family over $\Z_{8191}$.
We observe that $(8191,7,1)$ difference families %R with parameters
over~$\Z_{8191}$ were already known. They were obtained
%by Greig~\cite{G},
in~\cite{G} using a modification of a
construction % method
due to Wilson~\cite{W}.

\vspace{3ex}
%==============================================================================%
%                                                                              %
%   5. Discussion and open problems                                            %
%                                                                              %
%==============================================================================%
\section{Discussion and open problems}
\label{sec:conc}

There is no good reason to believe that many $q$-Steiner systems,
%R with~parameters other than those of $\dS_2[2,3,13]$,
other~than $\dS_2[2,3,13]$,
would not exist. % Specifically,
In particular, we~conjecture as follows.
\begin{conjecture}
%R $q$-Steiner systems $\dS_2[2,3,n]$ exist at least whenever
%R $n$ is a prime such that\/ $n \equiv 1 \pmod 6$ and $n \geq 13$.
If\/ $n \geq 13$ is a prime such that\/ $n \equiv 1\! \pmod 6$
then there exists a~$q$-Steiner system $\dS_2[2,3,n]$.
\end{conjecture}
\looseness=-1

The apparent large number of isomorphism classes of $\dS_2[2,3,13]$
$q$-Steiner systems suggests that an $\dS_2[3,4,14]$ $q$-Steiner system
might exist. A more~general open question is whether nontrivial
$q$-Steiner systems $\dS_q[t,k,n]$ exist for parameters
other than $q=2$, $t=2$, and $k=3$.

In fact, in light of our results, the main question is
no longer whether \mbox{$q$-Steiner} systems exist but rather
how they can be found.
Not only should~com\-puter-aided searches
be carried out, but one should also consider %R the possibility of
algebraic~and combinatorial constructions of either specific
$q$-Steiner systems or even %as well as
infinite families
of $q$-Steiner systems
%R (for example, in the framework of difference sets).
(e.g., in the framework of difference sets). %R families).

\vspace{3.00ex}
%==============================================================================%
%                                                                              %
%   6. Acknowledgements                                                        %
%                                                                              %
%==============================================================================%
\section*{Acknowledgments}

\noindent
The authors are grateful to
%Elisa Gorla, Joachim Rosenthal, and Amin Shokrollahi for organizing
the organizers of
the conference {\sc ``Trends in Coding Theory,''} held in
Ascona, Switzerland, between October 28, 2012, and November~2, 2012, where the
final pieces of this work were put together. The COST Action IC1104,
``Random network coding and designs over GF($q$)'' gave some of the motivation
for gathering four of the authors together in Ascona. The authors also thank Don
Knuth for making available his dancing links software. Last, but certainly not
the least we are grateful to Eimear Byrne who checked our results and found
that we wrote the wrong primitive polynomial in an earlier version.

\vspace{4.00ex}
%==============================================================================%
%                                                                              %
%   References                                                                 %
%                                                                              %
%==============================================================================%
\bibliographystyle{plain}

\end{document}